\documentclass[12pt]{article}
\usepackage{amssymb}
\usepackage{amsmath,amsthm}
\usepackage[latin1]{inputenc}
\usepackage[dvips]{graphicx}
\usepackage{hyperref}
\usepackage{enumerate}

\hypersetup{colorlinks=true}

\hypersetup{colorlinks=true, linkcolor=blue, citecolor=blue,urlcolor=blue}

 \newtheorem{remark}{Remark}

 \newtheorem{lemma}[remark]{Lemma}
 \newtheorem{theorem}[remark]{Theorem}
 
 \newtheorem{corollary}[remark]{Corollary}

\title{On the strong metric generators of strong product graphs}

\date{}

\author{Dorota Kuziak$^{(1)}$, Ismael G. Yero$^{(2)}$ and Juan A. Rodr\'{\i}guez-Vel\'{a}zquez$^{(1)}$ \\
$^{(1)}${\small Departament d'Enginyeria Inform\`atica i Matem\`atiques,}\\
{\small Universitat Rovira i Virgili,} {\small Av. Pa\"{\i}sos Catalans 26, 43007 Tarragona, Spain.}\\
{\small dorota.kuziak\@@urv.cat, juanalberto.rodriguez\@@urv.cat}\\
$^{(2)}${\small Departamento de Matem\'aticas, Escuela Polit\'ecnica Superior de Algeciras}\\
{\small Universidad de C\'adiz,} {\small Av. Ram\'on Puyol s/n, 11202 Algeciras, Spain.}\\
{\small ismael.gonzalez\@@uca.es}\\
}

\begin{document}
\maketitle

\begin{abstract}
Let $G$ be a connected graph. A vertex $w\in V(G)$ strongly resolves two vertices $u,v\in V(G)$ if there exists some shortest $u-w$ path containing $v$ or some shortest $v-w$ path containing $u$. A set $S$ of vertices is a strong metric generator for $G$ if every pair of vertices of $G$ is strongly resolved by some vertex of $S$. The smallest cardinality of a strong metric generator for $G$ is called the strong metric dimension of $G$. It is well known that the problem of computing this invariant is NP-hard. In this paper we study the problem of finding exact values or sharp bounds for the strong metric dimension of strong product graphs and express
these in terms of invariants of the factor graphs.
\end{abstract}



\section{Introduction}

A {\em generator} of a metric space is a set $S$ of points in the space with the property that every point of the space is uniquely determined by its distances from the elements of $S$. Given a simple and connected graph $G=(V,E)$, we consider the metric $d_G:V\times V\rightarrow \mathbb{R}^+$, where $d_G(x,y)$ is the length of a shortest path between $x$ and $y$. $(V,d_G)$ is clearly a metric space. A vertex $v\in V$ is said to distinguish two vertices $x$ and $y$ if $d_G(v,x)\ne d_G(v,y)$. A set $S\subset V$ is said to be a \emph{metric generator} for $G$ if any pair of vertices of $G$ is
distinguished by some element of $S$. A minimum generator is called a \emph{metric basis}, and its cardinality the \emph{metric dimension} of $G$. Motivated by the problem of uniquely determining the location of an intruder in a network, the concept of metric dimension of a graph was introduced by Slater in \cite{leaves-trees}, where the metric generators were called \emph{locating sets}. The concept of metric dimension of a graph was introduced independently by Harary and Melter in \cite{harary}, where metric generators were called \emph{resolving sets}.

In \cite{Sebo} a more restricted invariant than the metric dimension is introduced. A vertex $w\in V(G)$ \emph{strongly resolves} two vertices $u,v\in V(G)$ if $d_G(w,u)=d_G(w,v)+d_G(v,u)$ or $d_G(w,v)=d_G(w,u)+d_G(u,v)$, \emph{i.e.}, there exists some shortest $w-u$ path containing $v$ or some shortest $w-v$ path containing $u$. A set $S$ of vertices in a connected graph $G$ is a \emph{strong metric generator} for $G$ if every two vertices of $G$ are strongly resolved by some vertex of $S$. The smallest cardinality of a strong metric generator of $G$ is called the \emph{strong metric dimension} and is denoted by $dim_s(G)$. A \emph{strong metric basis} of $G$ is a strong metric generator for $G$ of cardinality $dim_s(G)$.

Given a simple graph $G=(V,E)$, we denote two adjacent vertices $u,v$ by $u\sim v$ and, in this case, we say that $uv$ is an edge of $G$, \emph{i.e.}, $uv\in E$. For a vertex $v\in V,$ the set $N(v)=\{u\in V: \; u\sim v\}$ is the open neighborhood of $v$ and the set $N[v] = N(v)\cup \{v\}$ is the closed neighborhood of $v$. The diameter of $G$ is defined as $D(G)=\max_{u,v\in V}\{d(u,v)\}$. The vertex  $x\in V$ is diametral in $G$ if there exists $y\in V$ such that $d_G(x,y)=D(G)$. We say that $G$ is $2$-antipodal if for each vertex $x\in V$ there exists exactly one vertex $y\in V$ such that $d_G(x,y)=D(G)$.

A set $S$ of vertices of $G$ is a \emph{vertex cover} of $G$ if every edge of $G$ is incident with at least one vertex of $S$. The \emph{vertex cover number} of $G$, denoted by $\alpha(G)$, is the smallest cardinality of a vertex cover of $G$. We refer to an $\alpha(G)$-set in a graph $G$ as a vertex cover set of cardinality $\alpha(G)$. A vertex $u$ of $G$ is \emph{maximally distant} from $v$ if for every $w\in N_G(u)$, $d_G(v,w)\le d_G(u,v)$. If $u$ is maximally distant from $v$ and $v$ is maximally distant from $u$, then we say that $u$ and $v$ are \emph{mutually maximally distant}. The {\em boundary} of $G=(V,E)$ is defined as $$\partial(G) = \{u\in V:\; \mbox{exists } v\in V\, \mbox{ such that } u,v\mbox{ are mutually maximally distant}\}.$$
In \cite{Oellermann} was first presented a relationship between the boundary of a graph and its strong metric generators. Also, in \cite{Oellermann} the authors defined the concept of \emph{strong resolving graph} $G_{SR}$ of a graph $G$ like a graph with vertex set $V(G_{SR}) = V$ and two vertices $u,v$ are adjacent in $G_{SR}$ if and only if $u$ and $v$ are mutually maximally distant in $G$.

We recall that the \textit{Cartesian product} of two graphs $G=(V_1,E_1)$ and $H=(V_2,E_2)$ is the graph $G\Box H=(V,E)$, such that $V=V_1\times V_2$ and two vertices $(a,b)\in V$ and $(c,d)\in V$ are adjacent in $G\Box H$ if and only if  either
\begin{itemize}
\item $a=c$ and $bd\in E_2$, or
\item $ac\in E_1$  and $b=d$.
\end{itemize}

The \textit{strong product} of two graphs $G=(V_1,E_1)$ and $H=(V_2,E_2)$ is the graph $G\boxtimes H=(V,E)$, such that
$V=V_1\times V_2$ and two vertices $(a,b)\in V$ and $(c,d)\in V$ are adjacent in $G\boxtimes H$ if and only if either
\begin{itemize}
\item $a=c$ and $bd\in E_2$, or
\item $ac\in E_1$ and $b=d$, or
\item $ac\in E_1$ and $bd\in E_2$.
\end{itemize}

The \textit{lexicographic product} of two graphs $G=(V_1,E_1)$ and $H=(V_2,E_2)$ is the graph $G\circ H$ with the vertex set $V=V_1\times V_2$ and two vertices $(a,b)\in V$ and $(c,d)$ are adjacent if either
\begin{itemize}
\item $ac\in E_1$, or
\item $a=c$ and $bd\in E_2$.
\end{itemize}

The \textit{Cartesian sum} of two graphs $G=(V_1,E_1)$ and $H=(V_2,E_2)$, denoted by $G\oplus H$, has as the vertex set $V=V_1\times V_2$ and two vertices $(a,b)\in V$ and $(c,d)\in V$ are adjacent in $G\oplus H$ if and only if $ac\in E_1$ or $bd\in E_2$. This notion of graph product was introduced by Ore \cite{Ore}. The Cartesian sum is also known as the disjunctive product \cite{disjunctive-product}.

Let $G=(V,E)$ and $G'=(V',E')$ be two graphs. If $V'\subseteq V$ and $E'\subseteq E$, then $G'$ is a subgraph of $G$ and we denote that by $G'\sqsubseteq G$.

In this article we are interested in the study of strong metric generators of strong product graphs. It was shown in \cite{Oellermann} that the problem of computing $dim_s(G)$ is NP-hard. This suggests obtaining exact values of the strong metric dimension for special classes of strong product graphs or finding sharp bounds on this invariant.

\section{Results}

Oellermann and Peters-Fransen \cite{Oellermann} showed that the problem of finding the strong metric dimension of a graph $G$ can be transformed into the problem of computing the vertex cover number of $G_{SR}$.

\begin{theorem}{\em \cite{Oellermann}}\label{th oellermann}
For any connected graph $G$,
$$dim_s(G) = \alpha(G_{SR}).$$
\end{theorem}

Recall that the largest cardinality of a set of vertices of $G$, no two of which are adjacent, is called the \emph{independence number} of $G$ and is denoted by $\beta(G)$. We refer to an $\beta(G)$-set in a graph $G$ as an independent set of cardinality $\beta(G)$. The following well-known result, due to Gallai, states the relationship between the independence number and the vertex cover number of a graph.

\begin{theorem}{\rm (Gallai's theorem)}\label{th gallai}
For any graph  $G$ of order $n$,
$$\alpha(G)+\beta(G) = n.$$
\end{theorem}
Thus, for any graphs $G$ and $H$ of order $n_1$ and $n_2$, respectively, by using Theorems \ref{th oellermann} and \ref{th gallai}, we immediately  obtain that

\begin{equation}
dim_s(G\boxtimes H) = n_1\cdot n_2 - \beta((G\boxtimes H)_{SR}) \label{oellermann-gallai}
\end{equation}

The following basic remark leads to a corollary about the neighborhood of a vertex in the strong product graph $G\boxtimes H$, which will be useful to present our results.

\begin{remark}\label{rem neighbor}
Let $G$ and $H$ be two graphs. For every $u\in V(G)$ and $v\in V(H)$ $$N_{G\boxtimes H}[(u,v)] = N_G[u]\times N_H[v].$$
\end{remark}

\begin{corollary}\label{cor neighbor}
Let $G$ and $H$ be two graphs and let $u,u'\in V(G)$ and $v,v'\in V(H)$.
The following assertion hold.
\begin{enumerate}[{\rm (i)}]
\item If $(u',v')\in N_{G\boxtimes H}(u,v)$, then $u'\in N_{G}[u]$ and $v'\in N_{G}[v]$.
\item If $u'\in N_{G}(u)$ and $v'\in N_{G}(v)$, then $(u',v')\in N_{G\boxtimes H}(u,v)$.
\end{enumerate}-
\end{corollary}

The following result about the boundary of strong product graphs was presented in \cite{boundary-in-strong}. Nevertheless in such a paper the authors are more interested into the cardinality of the boundary $\partial(G\boxtimes H)$ than into how the subgraph induced by boundary looks like.

\begin{theorem}{\em \cite{boundary-in-strong}}\label{th boundary}
For any graphs $G$ and $H$, $\partial(G\boxtimes H) = (\partial(G)\times V(H))\cup (V(G)\times \partial(H))$.
\end{theorem}

In the next lemma we pretend to describe the structure of the strong resolving graph of $G\boxtimes H$.

\begin{lemma}\label{lem boundary}
Let $G$ and $H$ be two connected nontrivial graphs. Let $u,x$ be two vertices of $G$ and let $v,y$ be two vertices of $H$. Then $(u,v)$ and $(x,y)$ are mutually maximally distant vertices in $G\boxtimes H$ if and only if one of the following conditions holds:
\begin{enumerate}[{\rm (i)}]
\item $u,x$ are mutually maximally distant in $G$ and $v,y$ are mutually maximally distant in $H$;
\item $u,x$ are mutually maximally distant in $G$ and $v=y$;
\item $v,y$ are mutually maximally distant in $H$ and $u=x$;
\item $u,x$ are mutually maximally distant in $G$ and $d_G(u,x) > d_H(v,y)$;
\item $v,y$ are mutually maximally distant in $H$ and $d_G(u,x) < d_H(v,y)$.
\end{enumerate}
\end{lemma}

\begin{proof}
(Sufficiency) Let $(u',v')\in N_{G\boxtimes H}(u,v)$ and $(x',y')\in N_{G\boxtimes H}(x,y)$. By Corollary \ref{cor neighbor} we have $u'\in N_G[u]$, $x'\in N_G[x]$, $v'\in N_H[v]$ and $y'\in N_H[y]$.

(i) If $u,x$ are mutually maximally distant in $G$ and $v,y$ are mutually maximally distant in $H$, then
$$d_{G\boxtimes H}((u',v'),(x,y)) = \max \{d_G(u',x), d_H(v',y)\}\le \max \{d_G(u,x), d_H(v,y)\} = d_{G\boxtimes H}((u,v),(x,y))$$ and
$$d_{G\boxtimes H}((u,v),(x,'y')) = \max \{d_G(u,x'), d_H(v,y')\}\le \max \{d_G(u,x), d_H(v,y)\} = d_{G\boxtimes H}((u,v),(x,y)).$$ Thus, $(u,v)$ and $(x,y)$ are mutually maximally distant vertices in $G\boxtimes H$.

(ii) If $u,x$ are mutually maximally distant in $G$ and $v=y$, then
$$d_{G\boxtimes H}((u',v'),(x,y)) = \max \{d_G(u',x), d_H(v',y)\} = d_G(u',x)\le d_G(u,x) = d_{G\boxtimes H}((u,v),(x,y))$$ and
$$d_{G\boxtimes H}((u,v),(x',y')) = \max \{d_G(u,x'), d_H(v,y')\} = d_G(u,x')\le d_G(u,x) = d_{G\boxtimes H}((u,v),(x,y)).$$ Thus, $(u,v)$ and $(x,y)$ are mutually maximally distant vertices in $G\boxtimes H$.

(iii) By using analogous procedure to (ii) we can show that if $u=x$ and $v,y$ are mutually maximally distant in $G$, then $(u,v)$ and $(x,y)$ are mutually maximally distant vertices in $G\boxtimes H$.

(iv) If $u,x$ are mutually maximally distant in $G$ and $d_G(u,x) > d_H(v,y)$, then
\begin{align*}
d_{G\boxtimes H}((u',v'),(x,y))& = \max \{d_G(u',x), d_H(v',y)\}\\
& \le \max \{d_G(u,x), d_H(v,y)+1\}\\
& = \max \{d_G(u,x), d_H(v,y)\}\\
& = d_{G\boxtimes H}((u,v),(x,y))
\end{align*}
and
\begin{align*}
d_{G\boxtimes H}((u,v),(x,'y'))& = \max \{d_G(u,x'), d_H(v,y')\}\\
& \le \max \{d_G(u,x), d_H(v,y)+1\}\\
& = \max \{d_G(u,x), d_H(v,y)\}\\
& = d_{G\boxtimes H}((u,v),(x,y)).
\end{align*}
Thus, $(u,v)$ and $(x,y)$ are mutually maximally distant vertices in $G\boxtimes H$.

(v) By using analogous procedure as in (iv) we can show that if $v,y$ are mutually maximally distant in $H$ and $d_G(u,x) < d_H(v,y)$, then $(u,v)$ and $(x,y)$ are mutually maximally distant vertices in $G\boxtimes H$.

(Necessity) Let $(u,v)$ and $(x,y)$ be two mutually maximally distant vertices in $G\boxtimes H$. Let $u'\in N_G(u)$, $x'\in N_G(x)$, $v'\in N_H(v)$ and $y'\in N_H(y)$. Notice that, by Corollary \ref{cor neighbor} $(u',v')\in N_{G\boxtimes H}(u,v)$ and $(x',y')\in N_{G\boxtimes H}(x,y)$. So, we have that
$$d_{G\boxtimes H}((u,v),(x,y))\ge d_{G\boxtimes H}((u',v'),(x,y))$$
and
$$d_{G\boxtimes H}((u,v),(x,y))\ge d_{G\boxtimes H}((u,v),(x',y')).$$
We differentiate two cases.

Case 1. $d_G(u,x)\ge d_H(v,y)$. Hence, $d_{G\boxtimes H}((u,v),(x,y)) = \max\{d_G(u,x),d_H(v,y)\} = d_G(u,x)$. Thus,
$$d_G(u,x)\ge \max\{d_G(u',x),d_H(v',y)\}$$
and
$$d_G(u,x)\ge \max\{d_G(u,x'),d_H(v,y')\}.$$
So, we obtain four inequalities:
\begin{equation}
d_G(u,x)\ge d_G(u',x),   \label{i1}
\end{equation}
\begin{equation}
d_G(u,x)\ge d_H(v',y),   \label{i2}
\end{equation}
\begin{equation}
d_G(u,x)\ge d_G(u,x'),   \label{i3}
\end{equation}
\begin{equation}
d_G(u,x)\ge d_H(v,y').   \label{i4}
\end{equation}

From (\ref{i1}) and (\ref{i3}) we have, that $u$ and $x$ are mutually maximally distant in $G$. If $v$ and $y$ are mutually maximally distant in $H$, then (i) holds and, if $v=y$, then (ii) holds. Suppose that there exists a vertex $v''\in N_H(v)$ such that $d_H(v'',y) > d_H(v,y)$ or there exists a vertex $y''\in N_H(y)$ such that $d_H(v,y'') > d_H(v,y)$. In such a case,
\begin{equation}
d_H(v'',y)\ge d_H(v,y)+1   \label{i1'}
\end{equation}
or
\begin{equation}
d_H(v,y'')\ge d_H(v,y)+1.   \label{i2'}
\end{equation}
Since $v''\in N_H(v)$, for any $u''\in N_G(u)$ we have $(u'',v'')\in N_{G\boxtimes H}(u,v)$ and following the above procedure, taking $(u'',v'')$ instead of $(u',v')$ we obtain two inequalities equivalent to (\ref{i2}) and (\ref{i4}). Thus,
\begin{equation}
d_G(u,x)\ge d_H(v'',y) > d_H(v,y)
\end{equation}
and
\begin{equation}
d_G(u,x)\ge d_H(v,y'') > d_H(v,y).
\end{equation}
So, $u,x$ are mutually maximally distant in $G$ and $d_G(u,x) > d_H(v,y)$. Hence, (iv) is satisfied.

Case 2. $d_G(u,x) < d_H(v,y)$. By using analogous procedure we can prove that $v,y$ are mutually maximally distant in $H$ and $u=x$ or $d_G(u,x) < d_H(v,y)$, showing that (iii) and (v) hold. Therefore, the result follows.
\end{proof}

Notice that Lemma \ref{lem boundary} leads to the following relationship.

\begin{theorem}\label{th products subgraphs}
For any connected graphs $G$ and $H$,
$$G_{SR}\boxtimes H_{SR}\sqsubseteq (G\boxtimes H)_{SR}\sqsubseteq G_{SR}\oplus H_{SR}.$$
\end{theorem}

\begin{proof}
Notice that $V(G_{SR}\boxtimes H_{SR})=V((G\boxtimes H)_{SR})=V(G_{SR}\oplus H_{SR})=V_1\times V_2$. Let $(u,v)$ and $(x,y)$ be two vertices adjacent in $G_{SR}\boxtimes H_{SR}$. So, either
\begin{itemize}
\item $u=x$ and $vy\in E(H_{SR})$, or
\item $ux\in E(G_{SR})$ and $v=y$, or
\item $ux\in E(G_{SR})$ and $vy\in E(H_{SR})$.
\end{itemize}
Hence, by using respectively the condition (iii), (ii) and (i) of Lemma \ref{lem boundary} we have that $(u,v)$ and $(x,y)$ are also adjacent in $(G\boxtimes H)_{SR}$.

Now, let $(u',v')$ and $(x',y')$ be two vertices adjacent in $(G\boxtimes H)_{SR}$. From Lemma \ref{lem boundary} we obtain that $u'x'\in E(G_{SR})$ or $v'y'\in E(H_{SR})$. Thus, $(u',v')$ and $(x',y')$ are also adjacent in $G_{SR}\oplus H_{SR}$.
\end{proof}

\begin{corollary}\label{cor beta products}
For any connected graphs $G$ and $H$,
$$\beta (G_{SR}\boxtimes H_{SR})\ge \beta ((G\boxtimes H)_{SR})\ge \beta(G_{SR}\oplus H_{SR}).$$
\end{corollary}

In order to better understand how the strong resolving graph $(G\boxtimes H)_{SR}$ looks like, by using Lemma \ref{lem boundary}, we prepare a kind of ``graphical representation'' of $(G\boxtimes H)_{SR}$ which we present in Figure \ref{graph SR}. According to the conditions (i), (ii) and (iii) of Lemma \ref{lem boundary} the solid lines represents those edges of $(G\boxtimes H)_{SR}$ which always exists. Also, from the conditions (iv) and (v) of Lemma \ref{lem boundary}, two vertices belonging to different rounded rectangles with identically filled areas could be adjacent or not in $(G\boxtimes H)_{SR}$.

\begin{figure}[h]
  \centering
  \includegraphics[width=0.65\textwidth]{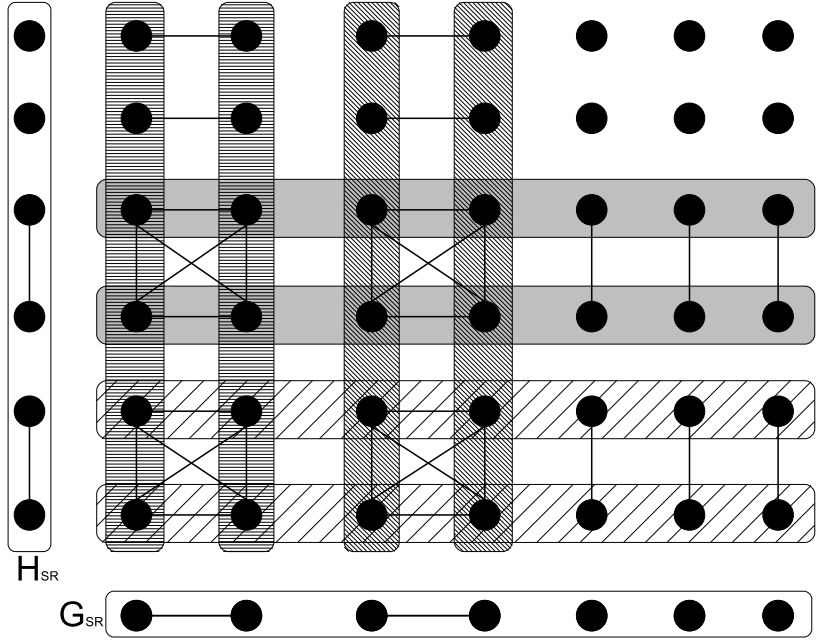}
  \caption{Sketch of a representation of a strong resolving graph $(G\boxtimes H)_{SR}$.}\label{graph SR}
\end{figure}

The following three known results will be useful for our purposes.

\begin{theorem}{\em \cite{product-independence}}\label{th products ind}
For any graphs $G$ and $H$,
$$\beta(G)\cdot \beta(H)\le \beta(G\boxtimes H)\le \beta(G\Box H).$$
\end{theorem}

\begin{theorem}{\rm (Vizing's theorem)}\label{th vizing}
For any graphs $G$ and $H$,
$$\beta(G\Box H)\le \min \{\beta(G)|V(H)|,\beta(H)|V(G)|\}.$$
\end{theorem}

\begin{theorem}{\em \cite{lexicographic-independence}}\label{th lexicographic ind}
For any graphs $G$ and $H$,
$$\beta(G\circ H) = \beta(G)\cdot \beta(H).$$
\end{theorem}

Next we present a lemma about the independence number of Cartesian sum graphs.

\begin{lemma}\label{lem Cartesian sum}
For any graphs $G$ and $H$,
$$\beta(G\oplus H) = \beta(G)\cdot \beta(H).$$
\end{lemma}

\begin{proof}
Let $A$ be a $\beta(G)$-set and let $B$ be a $\beta(H)$-set. Let $S = A\times B$. We will show that $S$ is an independent set in $G\oplus H$. Notice that if $|S|=1$, then $G$, $H$ and $G\oplus H$ are complete graphs and the result follows. Now, we consider case $|S|>1$.
Let $(u,v),(x,y)\in S$. Suppose that $(u,v)$ and $(x,y)$ are adjacent in $G\oplus H$. If $u=x$, then $v$ and $y$ are adjacent in $H$, which is a contradiction. If $v=y$, then analogously we have a contradiction. Now, if $u\ne x$ and $v\ne y$, then $u$ and $x$ are adjacent in $G$ or $v$ and $y$ are adjacent in $H$, which is a contradiction. Thus, $S$ is independent and we have that $\beta(G\oplus H)\ge \beta(G)\cdot \beta(H)$.

On the other hand, from the definitions of Cartesian sum and lexicographic product we have that $V(G\circ H)=V(G\oplus H)$ and $G\circ H\sqsubseteq G\oplus H$. Thus, $\beta(G\circ H)\ge \beta(G\oplus H)$ and by using Theorem \ref{th lexicographic ind} we have that $\beta(G\oplus H)\le \beta(G)\cdot \beta(H)$. Therefore, the result follows.
\end{proof}

\begin{theorem}\label{th bounds}
Let $G$ and $H$ be two connected nontrivial graphs of order $n_1$, $n_2$, respectively. Then
$$\max \{n_2\cdot dim_s(G),n_1\cdot dim_s(H)\} \le dim_s(G\boxtimes H)\le n_2\cdot dim_s(G) + n_1\cdot dim_s(H) - dim_s(G)\cdot dim_s(H).$$
\end{theorem}

\begin{proof}
By using Corollary \ref{cor beta products} we have that $\beta(G_{SR}\boxtimes H_{SR})\ge \beta((G\boxtimes H)_{SR})$. Hence, from equality (\ref{oellermann-gallai}), Theorem \ref{th products ind} and Theorem \ref{th vizing} we obtain
\begin{align*}dim_s(G\boxtimes H)& = n_1\cdot n_2 - \beta((G\boxtimes H)_{SR})\\
&\ge n_1\cdot n_2 - \beta(G_{SR}\boxtimes H_{SR})\\
&\ge n_1\cdot n_2 - \beta(G_{SR}\Box H_{SR})\\
&\ge n_1\cdot n_2 - \min\{n_2\cdot \beta(G_{SR}),n_1\cdot \beta(H_{SR})\}\\
& = \max\{n_2(n_1 - \beta(G_{SR})),n_1(n_2 - \beta(H_{SR}))\}\\
& = \max\{n_2\cdot dim_s(G),n_1\cdot dim_s(H)\}.
\end{align*}

On the other hand, from Corollary \ref{cor beta products} it follows $\beta((G\boxtimes H)_{SR})\ge \beta(G_{SR}\oplus H_{SR})$. So, by using (\ref{oellermann-gallai}) and Lemma \ref{lem Cartesian sum} we have
\begin{align*}dim_s(G\boxtimes H)& = n_1\cdot n_2 - \beta((G\boxtimes H)_{SR})\\
&\le n_1\cdot n_2 - \beta(G_{SR}\oplus H_{SR})\\
& = n_1\cdot n_2 - \beta(G_{SR})\cdot \beta(H_{SR})\\
& = n_1\cdot n_2 - (n_1 - dim_s(G))\cdot (n_2 - dim_s(H))\\
& = n_2\cdot dim_s(G) + n_1\cdot dim_s(H) - dim_s(G)\cdot dim_s(H).
\end{align*}
\end{proof}

We define a \emph{$\mathcal{C}$-graph} as a graph $G$ whose vertex set can be partitioned into $\beta(G)$ cliques. Notice that there are several graphs which are $\mathcal{C}$-graphs. For instance, we emphasize the following cases: complete graphs and cycles of even order. In order to prove the next result we also need to introduce the following notation. Given two graphs $G=(V_1,E_1)$, $H=(V_2,E_2)$ and a subset $X$ of vertices of $G\boxtimes H=(V,E)$, the \emph{projections} of $X$ over the graphs $G$ and $H$, respectively, are the following ones
$$P_G(X)=\{u\in V_1\,:\,(u,v)\in X,\,\mbox{ for some } v\in V_2\},$$
$$P_H(X)=\{v\in V_2\,:\,(u,v)\in X,\,\mbox{ for some } u\in V_1\}.$$

\begin{lemma}\label{lem Cgraph}
For any $\mathcal{C}$-graph $G$ and any graph $H$,
$$\beta(G\boxtimes H) = \beta(G) \cdot \beta(H).$$
\end{lemma}

\begin{proof}
Let $A_1, A_2, ... ,A_{\beta(G)}$ be a partition of $V(G)$ such that $A_i$ is a clique for every $i\in \{1,2,...,\beta(G)\}$. Let $S$ be an $\beta(G\boxtimes H)$-set and let $S_i = S\cap (A_i \times V_2)$ for $i\in \{1,2,...,\beta(G)\}$. First we will show that $P_H(S_i)$ is an independent set in $H$. If $|P_H(S_i)|=1$, then $P_H(S_i)$ is an independent set in $H$. If $|P_H(S_i)|\ge 2$, then for any two vertices $x,y\in P_H(S_i)$ there exist $u,v\in A_i$ such that $(u,x),(v,y)\in S_i$. We suppose that $x\sim y$. If $u=v$, then $(u,x)\sim (v,y)$, which is a contradiction. Thus, $u\ne v$. Since $(u,x)\not\sim (v,y)$, we have that $u\not\sim v$, which is a contradiction with the fact that $A_i$ is a clique. Therefore, for every $i\in \{1,2,...,\beta(G)\}$ the projection $P_H(S_i)$ is an independent set in $H$ and $\beta(H)\ge |P_H(S_i)|$.

Now, if $|S_i|>|P_H(S_i)|$ for some $i\in \{1,2,...,\beta(G)\}$, then there exists a vertex $z\in P_H(S_i)$ and two different vertices $a,b\in A_i$ such that $(a,z),(b,z)\in S_i$, and this is a contradiction with the facts that $A_i$ is a clique and $S_i$ is an independent set. Thus, $|S_i| = |P_H(S_i)|$, $i\in \{1,2,...,\beta(G)\}$, and we have the following
$$\beta(G\boxtimes H) = |S| = \sum_{i=1}^{\beta(G)}|S_i| = \sum_{i=1}^{\beta(G)}|P_H(S_i)|\le \beta(G)\cdot \beta(H).$$
Therefore, by using Theorem \ref{th products ind} we conclude the proof.
\end{proof}

\begin{theorem}\label{th Cgraphs value}
Let $G$ and $H$ be two connected nontrivial graphs of order $n_1$, $n_2$, respectively. If $G_{SR}$ is a $\mathcal{C}$-graph, then
$$dim_s(G\boxtimes H) = n_2\cdot dim_s(G) + n_1\cdot dim_s(H) - dim_s(G)\cdot dim_s(H).$$
\end{theorem}

\begin{proof}
By using Corollary \ref{cor beta products} we have that $\beta(G_{SR}\boxtimes H_{SR})\ge \beta((G\boxtimes H)_{SR})$. Hence, from equality (\ref{oellermann-gallai}) and Lemma \ref{lem Cgraph} we have
\begin{align*}dim_s(G\boxtimes H)& = n_1\cdot n_2 - \beta((G\boxtimes H)_{SR})\\
&\ge n_1\cdot n_2 - \beta(G_{SR}\boxtimes H_{SR})\\
& = n_1\cdot n_2 - \beta(G_{SR})\cdot \beta(H_{SR})\\
& = n_1\cdot n_2 - (n_1 - dim_s(G))\cdot (n_2 - dim_s(H))\\
& = n_2\cdot dim_s(G) + n_1\cdot dim_s(H) - dim_s(G)\cdot dim_s(H).
\end{align*}
The result now follows from Theorem \ref{th bounds}.
\end{proof}

A {\em cut vertex} in a graph is a vertex whose removal increases the number of connected component and a {\em simplicial vertex} is a vertex $v$ such that the subgraph induced by $N[v]$ is isomorphic to a complete graph. Also, a {\em block} is a maximal biconnected subgraph of the graph. Now, let $\mathfrak{F}$ be the family of sequences of connected graphs $G_1,G_2,...,G_k$, $k\ge 2$, such that $G_1$ is a complete graph $K_{n_1}$, $n_1\ge 2$, and $G_i$, $i\ge 2$, is obtained recursively from $G_{i-1}$ by adding a complete graph $K_{n_i}$, $n_i\ge 2$, and identifying a vertex of $G_{i-1}$ with a vertex in $K_{n_i}$.

From this point we will say that a connected graph $G$ is a \emph{generalized tree}\footnote{In some works those graphs are called block graphs.} if and only if there exists a sequence $\{G_1,G_2,...,G_k\}\in \mathfrak{F}$ such that $G_k=G$ for some $k\ge 2$. Notice that in these generalized trees every vertex is either, a cut vertex or a simplicial vertex. Also, every complete graph used to obtain the generalized tree is a block of the graph. Note that if every $G_i$ is isomorphic to $K_2$, then $G_k$ is a tree, justifying the terminology used.

At next we give examples of graphs for which its strong resolving graphs are $\mathcal{C}$-graphs.

\begin{itemize}
  \item $(K_n)_{SR}$ is isomorphic to $K_n$.
  \item For any complete $k$-partite graph such that at least all but one $p_i\ge 2$, $i\in\{1,2,...,k\}$, $(K_{p_1,p_2,...,p_k})_{SR}$ is isomorphic to the graph $\bigcup_{i=1}^{k}K_{p_i}$.
  \item If $G$ is a generalized tree of order $n$ and $c$ cut vertices, then $G_{SR}$ is isomorphic to the graph $K_{n-c}\cup \left(\bigcup_{i=1}^{c}K_1\right)$.
  \item For any $2$-antipodal\footnote{Notice that for instance cycles of even order are $2$-antipodal graphs.} graph $G$ of order $n$, $G_{SR}$ is isomorphic to the graph $\bigcup_{i=1}^{\frac{n}{2}}K_2$.
  \item For any grid graph, $(P_n\Box P_r)_{SR}$ is isomorphic to the graph $K_4\cup \left(\bigcup_{i=1}^{n\cdot r-4}K_1\right)$.
\end{itemize}

By using the above examples and Theorem \ref{th Cgraphs value} we have the following corollary.

\begin{corollary}\label{cor Cgraphs}
Let $G$ and $H$ be two connected nontrivial graphs of order $n_1$ and $n_2$, respectively.
\begin{enumerate}[{\rm(i)}]
  \item $dim_s(K_{n_1}\boxtimes H) = n_2(n_1 - 1) + n_1\cdot dim_s(H) - (n_1 - 1)dim_s(H)$.
  \item If $G$ is a complete $k$-partite graph, then $$dim_s(G\boxtimes H) = n_2(n_1 - k) + n_1\cdot dim_s(H) - (n_1 - k)dim_s(H).$$
  \item If $G$ is a generalized tree with $c$ cut vertices, then $$dim_s(G\boxtimes H) = n_2(n_1 - c - 1) + n_1\cdot dim_s(H) - (n_1 - c - 1)dim_s(H).$$ Particularly, if $G$ is a tree with $l(G)$ leaves, then $$dim_s(G\boxtimes H) = n_2(l(G) - 1) + n_1\cdot dim_s(H) - (l(G) - 1)dim_s(H).$$
  \item If $G$ is a $2$-antipodal graph, then $$dim_s(G\boxtimes H) = \frac{n_2\cdot n_1}{2} + n_1\cdot dim_s(H) - \frac{n_1}{2}\cdot dim_s(H).$$
  \item If $G$ is a grid graph, then $$dim_s(G\boxtimes H) = 3n_2 + n_1\cdot dim_s(H) - 3dim_s(H).$$
\end{enumerate}
\end{corollary}

Notice that Corollary \ref{cor Cgraphs} (iv) gives the value for the strong metric dimension of $C_r\boxtimes H$ for any graph $H$ and $r$ even. Next we study separately the strong product graphs $C_r\boxtimes H$ for any graph $H$ and $r$ odd. In order to prove the next result we need to introduce the following notation. We define a \emph{$\mathcal{C}_1$-graph} as a graph $G$ whose vertex set can be partitioned into $\beta(G)$ cliques and one isolated vertex. Notice that cycles with odd order are $\mathcal{C}_1$-graphs.

\begin{lemma}\label{lem C1graph}
For any $\mathcal{C}_1$-graph $G$ and any graph $H$,
$$\beta(G\boxtimes H)\le \beta(G)(\beta(H)+1).$$
\end{lemma}

\begin{proof}
Let $A_1, A_2, ... ,A_{\beta(G)}, B$ be a partition of $V(G)$ such that $A_i$ is a clique for every $i\in \{1,2,...,\beta(G)\}$ and $B=\{b\}$, where $b$ is isolated vertex. Let $S$ be an $\beta(G\boxtimes H)$-set and let $S_i = S\cap (A_i\times V_2)$ and $i\in \{1,2,...,\beta(G)\}$. Let $S_B = S\cap (B\times V_2)$. By using analogous procedures as in proof of Lemma \ref{lem Cgraph} we can show that for every $i\in \{1,2,...,\beta(G)\}$, $P_H(S_i)$ is an independent set in $H$ and $|S_i| = |P_H(S_i)|$. Moreover, since $|B| = 1$ we have that $P_H(S_B)$ is an independent set in $H$ and $|S_B| = |P_H(S_B)|$. Thus, we obtain the following
$$\beta(G\boxtimes H) = |S| = \sum_{i=1}^{\beta(G)}|S_i| + |S_B| = \sum_{i=1}^{\beta(G)}|P_H(S_i)| + |P_H(S_B)|\le \beta(G)\cdot \beta(H) + \beta(H) = \beta(G)(\beta(H)+1).$$
\end{proof}

\begin{theorem}\label{th C1graphs bounds}
Let $G$ and $H$ be two connected nontrivial graphs of order $n_1$, $n_2$, respectively. If $G_{SR}$ is a $\mathcal{C}_1$-graph, then
$$dim_s(G\boxtimes H)\ge n_1(dim_s(H) - 1) + dim_s(G)(n_2 - dim_s(H) + 1).$$
\end{theorem}

\begin{proof}
By using Corollary \ref{cor beta products} we have that $\beta(G_{SR}\boxtimes H_{SR})\ge \beta((G\boxtimes H)_{SR})$. Hence, from equality (\ref{oellermann-gallai}) and Lemma \ref{lem C1graph} we have
\begin{align*}dim_s(G\boxtimes H)& = n_1\cdot n_2 - \beta((G\boxtimes H)_{SR})\\
&\ge n_1\cdot n_2 - \beta(G_{SR}\boxtimes H_{SR})\\
&\ge n_1\cdot n_2 - \beta(G_{SR})(\beta(H_{SR})+1)\\
& = n_1\cdot n_2 - (n_1 - dim_s(G))\cdot (n_2 - dim_s(H) + 1)\\
& = n_1(dim_s(H) - 1) + dim_s(G)(n_2 - dim_s(H) + 1).
\end{align*}
\end{proof}

Since $dim_s(C_{2r+1}) = r + 1$, Theorems \ref{th bounds} and \ref{th C1graphs bounds} lead to the following result.
\begin{theorem}\label{th Cood H bounds}
Let $H$ be a connected nontrivial graphs of order $n$ and $r\ge 1$. Then
$$n(r + 1) + r(dim_s(H) - 1)\le dim_s(C_{2r+1}\boxtimes H)\le n(r + 1) + r\cdot dim_s(H).$$
\end{theorem}
The independence number of $C_{2r+1}\boxtimes C_{2t+1}$ was studied in \cite{hales}. There was presented the following result.

\begin{theorem}{\em \cite{hales}}\label{th Codd Codd indep}
For $1\le r\le t$,
$$\beta(C_{2r+1}\boxtimes C_{2t+1}) = r\cdot t + \left\lfloor \frac{r}{2} \right\rfloor.$$
\end{theorem}
By using the above result we obtain the following.

\begin{theorem}\label{th Codd Cood bounds}
For $1\le r\le t$,
$$3rt + 2r + 2t +1 - \left\lfloor \frac{r}{2} \right\rfloor \le dim_s(C_{2r+1}\boxtimes C_{2t+1})\le 3rt + 2r + 2t +1.$$
\end{theorem}

\begin{proof}
By using Theorem \ref{th products subgraphs} we have that $G_{SR}\boxtimes H_{SR}\sqsubseteq (G\boxtimes H)_{SR}$. Thus, $\beta(G_{SR}\boxtimes H_{SR})\ge \beta((G\boxtimes H)_{SR})$. Hence, from equality (\ref{oellermann-gallai}) and Theorem \ref{th Codd Codd indep} we have
\begin{align*}dim_s(C_{2r+1}\boxtimes C_{2t+1})& = (2r+1)\cdot (2t+1) - \beta((C_{2r+1}\boxtimes C_{2t+1})_{SR})\\
&\ge (2r+1)\cdot (2t+1) - \beta((C_{2r+1})_{SR}\boxtimes (C_{2t+1})_{SR})\\
& = (2r+1)\cdot (2t+1) - \beta(C_{2r+1}\boxtimes C_{2t+1})\\
& = (2r+1)\cdot (2t+1) - r\cdot t - \left\lfloor \frac{r}{2} \right\rfloor\\
& = 3rt + 2r + 2t +1 - \left\lfloor \frac{r}{2} \right\rfloor.
\end{align*}
The upper bound is direct consequence of Theorem \ref{th C1graphs bounds}.
\end{proof}

Notice that for $r = 1$ the lower bound is equal to the upper bound in the above theorem. Thus, $dim_s(C_{3}\boxtimes C_{2t+1}) = 5t + 3$ for every $t\ge 1$.

\end{document}